\documentclass[a4paper,10pt,reqno]{amsart}

\usepackage[utf8]{inputenc}
\usepackage[T1]{fontenc}
\usepackage{amsthm}
\usepackage{amsmath}
\usepackage{amssymb}
\usepackage[inline]{enumitem}
\usepackage{comment}
\PassOptionsToPackage{hyphens}{url}\usepackage{hyperref}
\usepackage{fancyhdr}
\usepackage{mathrsfs}
\usepackage{stmaryrd}
\usepackage[normalem]{ulem}
\usepackage{xcolor}
\usepackage{tikz-cd}
\usetikzlibrary{arrows}

\usepackage{cite}
\setlist[description]{%
  itemsep=0.05cm,               % space between items
  font={\normalfont\textsc}, % set the label font
 leftmargin=\parindent,
 labelindent=\parindent
}

\theoremstyle{definition}

\newtheorem{theorem}{Theorem}[section]
\newtheorem{lemma}[theorem]{Lemma}
\newtheorem{proposition}[theorem]{Proposition}
\newtheorem{corollary}[theorem]{Corollary}

\theoremstyle{definition}
\newtheorem{definition}[theorem]{Definition}

\newtheorem{remark}[theorem]{Remark}
\newtheorem{question}[theorem]{Question}
\newtheorem{questions}[theorem]{Questions}

\newtheorem{remarks}[theorem]{Remarks}

\newtheorem*{conjecture*}{Conjecture}
\newtheorem*{problem*}{Problem}
\newtheorem*{question*}{Question}

\definecolor{blue-url}{RGB}{0,0,100}
\definecolor{red-url}{RGB}{100,0,0}
\definecolor{green-url}{RGB}{0,100,0}
\definecolor{light-yellow}{RGB}{255,255,128}
\definecolor{light-blue}{RGB}{193,255,255}
\definecolor{light-red}{RGB}{239,83,80}

\hypersetup{
	pdftitle={Groups are globally closed, and so are numerical monoids},
	pdfauthor={},
	pdfmenubar=false,
	pdffitwindow=true,
	pdfstartview=FitH,
	colorlinks=true,
	linkcolor=blue-url,
	citecolor=green-url,
	urlcolor=red-url
}

\renewcommand{\emptyset}{\varnothing}
\renewcommand{\setminus}{\smallsetminus}
\renewcommand{\,}{\kern 0.1em}

\providecommand\llb{\llbracket}
\providecommand\rrb{\rrbracket}

\providecommand\dinf{\mathsf{d}_\ast}

\newcommand{\evid}[1]{\textsf{#1}}
\newcommand{\fin}{\mathrm{fin}}

\newcommand{\dil}[2]{#1 \times #2}
	
{\newline\vspace{\abovedisplayskip}\hbox to \textwidth\bgroup\hss$\displaystyle}
{$\hss\egroup\vspace{\belowdisplayskip}}

\makeatletter
\DeclareFontFamily{OMX}{MnSymbolE}{}
\DeclareSymbolFont{MnLargeSymbols}{OMX}{MnSymbolE}{m}{n}
\SetSymbolFont{MnLargeSymbols}{bold}{OMX}{MnSymbolE}{b}{n}
\DeclareFontShape{OMX}{MnSymbolE}{m}{n}{
	<-6>  MnSymbolE5
	<6-7>  MnSymbolE6
	<7-8>  MnSymbolE7
	<8-9>  MnSymbolE8
	<9-10> MnSymbolE9
	<10-12> MnSymbolE10
	<12->   MnSymbolE12
}{}
\DeclareFontShape{OMX}{MnSymbolE}{b}{n}{
	<-6>  MnSymbolE-Bold5
	<6-7>  MnSymbolE-Bold6
	<7-8>  MnSymbolE-Bold7
	<8-9>  MnSymbolE-Bold8
	<9-10> MnSymbolE-Bold9
	<10-12> MnSymbolE-Bold10
	<12->   MnSymbolE-Bold12
}{}

\let\llangle\@undefined
\let\rrangle\@undefined
\DeclareMathDelimiter{\llangle}{\mathopen}%
{MnLargeSymbols}{'164}{MnLargeSymbols}{'164}
\DeclareMathDelimiter{\rrangle}{\mathclose}%
{MnLargeSymbols}{'171}{MnLargeSymbols}{'171}
\makeatother
\begin{document}
\title{On global isomorphisms and a closure property of semigroups}

\author{Lingxi Li}
\address{(L.~Li) School of Mathematical Sciences, Hebei Normal University | Shijiazhuang, Hebei province, 050024 China}
\email{lingxi.li@qq.com}

\author{Salvatore Tringali}
\address{(S.~Tringali) School of Mathematical Sciences, Hebei Normal University | Shijiazhuang, Hebei province, 050024 China}
\email{salvo.tringali@gmail.com}
%\urladdr{http://imsc.uni-graz.at/tringali}

%Primary 11B13. Secondary 20M13, 39B52}
%
% 05E16: Combinatorial aspects of groups and algebras
% 11B30: arithmetic combinatorics
% 11P99: general additive NT
% 20E34: General structure theorems for groups
% 20M10: "General structure theory for semigroups" 
% 20M13: arithmetic of sgrps


\begin{abstract}
Let $S$ be a semigroup (written multiplicatively). Endowed with the operation of setwise mul\-ti\-pli\-ca\-tion induced by $S$ on its parts, the non-empty subsets of $S$ form themselves a sem\-i\-group, de\-not\-ed by $\mathcal P(S)$. Accordingly, we say that a semigroup $H$ is globally isomorphic to a semigroup $K$ if $\mathcal P(H)$ is isomorphic to $\mathcal P(K)$; and that a class $\mathscr C$ of sem\-i\-groups is globally closed if a semigroup in $\mathscr C$ can only be globally isomorphic to an isomorphic copy of a semigroup in the same class.

We show that the classes of groups, torsion-free monoids, and numerical monoids are each globally closed. The first result extends a 1967 theorem of Shafer, while the last relies non-trivially on the second and on a classical theorem of Kneser from additive number theory.
\end{abstract}

\maketitle

\thispagestyle{empty}

\section{Introduction}
\label{sect:1}

Let $S$ be a semigroup (see the end of this section for notations and terminology). Endowed with the binary operation of setwise multiplication induced by $S$ on its parts and defined by
\[
XY := \{xy : x \in X, \, y \in Y\}, \qquad\text{for all } X, Y \subseteq S,
\]
the \textit{non-empty} subsets of $S$ form a semigroup in their own right, herein denoted by $\mathcal P(S)$ and called the \evid{large power semigroup} of $S$. Furthermore, the family of all non-empty \textit{finite} subsets of $S$ is a sub\-sem\-i\-group of $\mathcal P(S)$, denoted by $\mathcal P_\fin(S)$ and called the \evid{finitary power semigroup} of $S$. 

In the sequel, we will generically refer to either $\mathcal P(S)$ or $\mathcal P_\fin(S)$ as a \evid{power sem\-i\-group}. 
These structures have played a pivotal role in
the ongoing development of an arithmetic theory of semigroups and rings 
\cite{Fa-Tr18, An-Tr18, Tr20(c), Tr21(b), Co-Tr-21(a), Co-Tr-22(a), Co-Tr-22(b)} that aims to extend the classical theory of factorization \cite{Ger-Hal-06} beyond its traditional boundaries. Moreover, they provide a natural algebraic framework for various problems in additive combinatorics and closely related areas \cite{Bie-Ger-22, Tri-Yan-23(a), Tri-Yan2023(b), GarSan-Tri-24(a)}, including S\'ark\"ozy's conjecture on the ``additive irreducibility'' of the set of [non-zero] squares of a finite field of prime order \cite[Conjecture 1.6]{Sark2012}.
%and Ostmann's conjecture \cite[p.~13]{Ostm-1968} on the ``asymptotic additive irreducibility'' of the set of (positive rational) primes.

It seems that power semigroups made their first explicit appearance in a 1953 paper by Dubreil \cite{Dubr-1953}, though Dubreil argued in favor of including the empty set (see the unnumbered remark from loc.~cit., p.~281).
%Their systematic investigation was initiated in the late 1960s and continued quite intensively by semigroup theorists and computer scientists throughout the 1980s and 1990s. 
%By that time, researchers were especially interested in properties of $S$ that do or do not ascend to $\mathcal P(S)$, as well as in the study of varieties (namely, classes of semigroups that are closed under taking homomorphic images, subsemigroups, and finite direct products) generated by $\mathcal P(S)$ as $S$ ranges over a specified family of finite semigroups. 
%A main catalyst of these early developments has been the role of power semigroups in the study of formal languages and automata (see the surveys by Pin \cite{Pin1995} and Almeida \cite{Alm02} for additional information and background).
A turning point in their history was marked by a paper of Tamura and Shafer \cite{Tam-Sha1967} that has eventually led to the following questions, where a sem\-i\-group $H$ is said to be \evid{globally isomorphic} to a semigroup $K$ if there is a \evid{global isomorphism} from $H$ to $K$, that is, an isomorphism $\mathcal P(H) \to \mathcal P(K)$.

\begin{questions}
\label{ques:tamura-shafer-iso-problem}
Let $\mathscr C$ be a class of semigroups. Given $H, K \in \mathscr C$, is it true that
\begin{enumerate}[label=\textup{(\alph{*})}]
\item\label{ques:tamura-shafer-iso-problem(1)} $H$ is globally isomorphic to $K$ if and only if $H$ is isomorphic to $K$?
\item\label{ques:tamura-shafer-iso-problem(2)} $\mathcal P_\fin(H)$ is isomorphic to $\mathcal P_\fin(K)$ if and only if $H$ is isomorphic to $K$?
\end{enumerate}
\end{questions}
The interesting aspect of these questions lies in the ``only if\,'' direction. In fact, if $f$ is an isomorphism from a semigroup $H$ to a semigroup $K$, then its \evid{augmentation} 
\begin{equation}\label{equ:augmentation}
f^\ast \colon \mathcal P(H) \to \mathcal P(K) \colon X \mapsto f[X] := \{f(x) \colon x \in X\}
\end{equation}
is a global isomorphism from $H$ to $K$.
Moreover, $\mathcal P_\fin(H)$ is isomorphic to $\mathcal P_\fin(K)$ via the restriction of $f^\ast$ to the non-empty finite subsets of $H$, since $f^\ast(X) \in \mathcal P_\fin(K)$ for every $X \in \mathcal P_\fin(H)$. See \cite[Remark 4]{Tri-2024(a)} and \cite[Sect.~1]{GarSan-Tri-24(a)} for further details and context.

The answer to Question \ref{ques:tamura-shafer-iso-problem}\ref{ques:tamura-shafer-iso-problem(1)} is negative for the class of all semigroups \cite{Mog1973}; is positive for groups \cite{Shaf-1967}, semilattices \cite[p.~218]{Koba-1984}, completely $0$-simple semigroups and completely simple semigroups \cite[Theorems 5.9 and 6.8]{Tamu-1986}, Clifford semigroups \cite[Theorem 4.7]{Gan-Zhao-2014}, cancellative commutative semigroups \cite[Corollary 1]{Tri-2024(a)}, etc.;
and is open for finite semigroups \cite[p.~5]{Hami-Nord-2009}, despite some authors having claimed the opposite based on results announced by Tamura in \cite{Tamu-1987} but never proved.

As for Question \ref{ques:tamura-shafer-iso-problem}\ref{ques:tamura-shafer-iso-problem(2)}, very little seems to be known outside the case where $\mathscr C$ is a class of \textit{finite} semigroups contained in any of the classes that are already addressed by the ``positive results'' reviewed in the previous paragraph (note that the large and finitary power semigroups of a semigroup $S$ coincide if and only if $S$ is finite). More precisely, Bienvenu and Geroldinger have shown in \cite[Theorem 3.2(3)]{Bie-Ger-22} that the problem has an affirmative answer for numerical monoids, and the conclusion has been subsequently generalized to cancellative commutative semigroups in \cite[Corollary 1]{Tri-2024(a)}. Here as usual, a \evid{numerical monoid} is a sub\-monoid of the additive monoid of non-negative integers with finite complement in $\mathbb N$.

Recent investigations have also addressed some variants of the above questions. Specifically, let $M$ be a monoid, with $1_M$ denoting its identity (element) and $M^\times$ its group of units. Then,  $\mathcal P(M)$ is itself a monoid and $\mathcal P_\fin(M)$ is a submonoid of $\mathcal P(M)$, their identity being the singleton $\{1_M\}$; we will accordingly call $\mathcal P(M)$ the \evid{large power monoid} of $M$.
In addition, each of the families
\[
\mathcal P_1(M) := \{X \in \mathcal P(M) \colon 1_M \in X\}
\qquad\text{and}\qquad
\mathcal P_\times(M) := \{X \in \mathcal P(M) : X \cap M^\times \ne \emptyset\}
\]
is a submonoid of $\mathcal P(M)$, and each of 
\[
\mathcal P_{\fin,1}(M) := \mathcal P_1(M) \cap \mathcal P_\fin(M)
%\{X \in \mathcal P_\fin(M) \colon 1_M \in X\}
\qquad\text{and}\qquad
\mathcal P_{\fin,\times}(M) := \mathcal P_\times(M) \cap \mathcal P_\fin(M)
%\{X \in \mathcal P_\fin(M) : X \cap M^\times \ne \emptyset\}
\]
is a submonoid of $\mathcal P_\fin(M)$.
%To our knowledge, $\mathcal{P}_1(M)$ made its earliest appearance in the literature in a 1984 paper by Margolis and Pin \cite{Marg-Pin-1984}, the definition of $\mathcal{P}_\times(M)$ is new, and the last two constructions were first considered by Fan and Tringali in \cite{Fa-Tr18}. 
Some basic relations among these structures are summarized in the diagram below, where a ``hooked arrow'' $P \hookrightarrow Q$ means the inclusion map from $P$ to $Q$ and a ``tailed arrow'' $P \rightarrowtail Q$ means the embedding $P \to Q \colon x \mapsto \{x\}$:
\vskip 0.2cm
\begin{center}
\begin{tikzpicture}[scale=1]
% nodes

\node (A1) at (0,0) {$\{1_M\}$};
\node (A2) at (2.6,0) {$M^\times$};
\node (A3) at (5.2,0) {$M$};

\node (B1) at (0,-1.3) {$\mathcal P_{\fin,1}(M)$};
\node (B2) at (2.6,-1.3) {$\mathcal P_{\fin,\times}(M)$};
\node (B3) at (5.2,-1.3) {$\mathcal P_\fin(M)$};

\node (C1) at (0,-2.6) {$\mathcal P_1(M)$};
\node (C2) at (2.6,-2.6) {$\mathcal P_\times(M)$};
\node (C3) at (5.2,-2.6) {$\mathcal P(M)$};

% arrow

\draw[right hook->, shorten <= 1pt, shorten >= 1pt] (A1) -- (A2);
\draw[>->, shorten <= 1pt, shorten >= 1pt] (A1) -- (B1);
\draw[right hook->, shorten <= 1pt, shorten >= 1pt] (A2) -- (A3);
\draw[>->, shorten <= 1pt, shorten >= 1pt] (A2) -- (B2);
\draw[>->, shorten <= 1pt, shorten >= 1pt] (A3) -- (B3);

\draw[right hook->, shorten <= 1pt, shorten >= 1pt] (B1) -- (B2);
\draw[right hook->, shorten <= 1pt, shorten >= 1pt] (B1) -- (C1);
\draw[right hook->, shorten <= 1pt, shorten >= 1pt] (B2) -- (B3);
\draw[right hook->, shorten <= 1pt, shorten >= 1pt] (B2) -- (C2);
\draw[right hook->, shorten <= 1pt, shorten >= 1pt] (B3) -- (C3);

\draw[right hook->, shorten <= 1pt, shorten >= 1pt] (C1) -- (C2);
\draw[right hook->, shorten <= 1pt, shorten >= 1pt] (C2) -- (C3);
\end{tikzpicture}
\end{center}

Here, we focus our attention on $\mathcal P_{\fin,1}(H)$, though some aspects of the theory (see, e.g., Corollary \ref{cor:finite-back-to-finite} and Remark \ref{remark:conjecture}) would benefit from a better understanding of the synergies among  
the objects in the second and third rows of the diagram.
Most notably, by analogy with Questions \ref{ques:tamura-shafer-iso-problem}, we ask the following:

\begin{question}\label{ques:BG-like-for-monoids}
Let $\mathscr C$ be a class of monoids. Is it true that $\mathcal P_{\fin,1}(H)$ is isomorphic to $\mathcal P_{\fin,1}(K)$, for some $H, K \in \mathscr C$, if and only if $H$ is isomorphic to $K$?
\end{question}

Once again, the core of Question \ref{ques:BG-like-for-monoids} lies in proving the ``only if\,'' direction. In fact, let $f$ be a semigroup iso\-mor\-phism from a monoid $H$ to a monoid $K$, and let $f^\ast$ be the augmentation of $f$, as defined by Eq.~\eqref{equ:augmentation}. Then, $f$ maps the identity of $H$ to the identity of $K$ (see, for instance, the last lines of \cite[Sect.~2]{Tri-2024(a)}), and hence $f(u) \in K^\times$ for every $u \in H^\times$. Since $f[X]$ is a finite subset of $K$ for every $X \in \mathcal P_\fin(X)$, it is then routine to check \cite[Remark 1.1]{Tri-Yan-23(a)} that $f^\ast$ restricts to an isomorphism from $\mathcal P_{\fin,1}(H)$ to $\mathcal P_{\fin,1}(K)$.

With these preliminaries in mind, Tringali and Yan \cite[Theorem 2.5]{Tri-Yan-23(a)} have proved that the answer to Question \ref{ques:BG-like-for-monoids} is positive for the class of (\evid{rational}) \evid{Puiseux monoids}, that is, submonoids of the non-negative rational numbers under addition; the result has confirmed a conjecture by Bienvenu and Geroldinger \cite[Conjecture 4.7]{Bie-Ger-22} on numerical monoids. More recently, Rago \cite{Rago} has shown that the answer to the same question is negative for the class of cancellative commutative monoids, by establishing more generally that if $H$ and $K$ are cancellative (commutative) valuation monoids \cite[Definition 2.6.4.2]{Halter-Koch} with trivial groups of units and isomorphic group of fractions, then $\mathcal P_{\fin,1}(H)$ is isomorphic to $\mathcal P_{\fin,1}(K)$. 

The present work contributes to this line of research by inquiring into a certain ``closure property'' of (classes of) semigroups. More precisely, we have the following:

\begin{definition}\label{def:globally-closed-classes}
A class $\mathscr C$ of semigroups is \evid{globally closed} if, whenever a semigroup $H \in \mathscr C$ is globally isomorphic to a semigroup $K$, there exists a semigroup $K' \in \mathscr C$ that is isomorphic to $K$.
\end{definition}

If the class $\mathscr C$ in Definition \ref{def:globally-closed-classes} is closed under isomorphisms, then being globally closed is equivalent to the property that a semigroup $H \in \mathscr C$ can only be globally isomorphic to another semigroup belonging to $\mathscr C$. For instance, it is obvious that the class of all semigroups is globally closed. Slightly more interestingly, the same is true of monoids \cite[Lemma 1.1]{Gou-Isk-Tsi-1984}. This leads us to the following:

\begin{question}
Is the class of cancellative semigroups globally closed?
\end{question}

The question is already challenging in the cancellative \textit{commutative} setting, and our main goal here is to show that the answer is positive for groups (Corollary \ref{cor:groups-are-globally-closed}) and numerical monoids (Theorem \ref{thm:numerical-monoids-are-globally-closed}).

It follows from the first of these results that if a group $H$ is globally isomorphic to a semigroup $K$, then $H$ is isomorphic to $K$ (and hence $K$ is itself a group). We can thus extend a 1967 theorem of Shafer \cite{Shaf-1967}, according to which two groups are globally isomorphic (that is, one is globally isomorphic to the other) if and only if they are isomorphic: the (non-trivial) difference here is that, in Shafer's half-page note, both $H$ and $K$ are assumed to be groups from the outset. 
\begin{comment}
Incidentally, the question of the global closure of groups was posed by user Micha{\l} Masny on MathOverflow in October 2012:
\begin{center}
\url{https://mathoverflow.net/questions/110412/#comment284325_110421}.
\end{center}
Benjamin Steinberg (CUNY Graduate Center, US) promptly provided an answer for \textit{finite} groups:
\begin{center}
\url{https://mathoverflow.net/a/110506/16537}.
\end{center}
Our proof of the general case takes a different path from Steinberg's, which allows us to establish as a bonus that
\end{comment}
As a byproduct of the proof, we obtain that every global isomorphism from a monoid $H$ to a monoid $K$ maps $H^\times$ to $K^\times$, and hence restricts to a global isomorphism from one group of units to the other (Theorem \ref{prop:global-iso-maps-unit-group-to-unit-group}).

As for the second result (Theorem \ref{thm:numerical-monoids-are-globally-closed}), we gather from \cite[Corollary 1]{Tri-2024(a)} that two cancellative commutative semigroups are globally isomorphic if and only if they are isomorphic. We are thus reduced to demonstrating that a numerical monoid can only be globally isomorphic to a cancellative commutative semigroup. Crucial to this end will be a classical theorem of Kneser \cite[Theorem I.17$^\prime$]{Halb-Roth-1967} on sets of non-negative integers satisfying a small doubling condition with respect to the lower asymptotic density (see Sect.~\ref{sec:additive_NT} and, in particular, Theorem \ref{thm:kneser-thm}). Another ingredient is Theorem \ref{thm:torsion-free-monoids-are-globally-closed}, where we establish that the class of torsion-free monoids is globally closed, and the heart of whose proof is essentially a rough counting argument for the solutions to certain equations involving idempotents.
 
\subsection*{Generalities.} We denote by $\mathbb N$ the (set of) non-negative integers, by $\mathbb N^+$ the positive integers, by $\mathbb Z$ the integers, by $|X|$ the cardinality of a set $X$, and by $f^{-1}$ the (functional) inverse of a bijection $f$.
Unless otherwise stated, we reserve the letters $m$ and $n$ (with or without subscripts) for positive integers.

If $a, b \in \mathbb Z$, we let $\llb a, b \rrb := \{x \in \allowbreak \mathbb Z \colon \allowbreak a \le x \le b\}$ be the (\evid{discrete}) \evid{interval} from $a$ to $b$. Given $k \in \mathbb N$ and $X \subseteq \mathbb Z$, we use $kX := \{x_1 + \cdots + x_k: x_1, \ldots, x_k \in X\}$ for the \evid{$k$-fold sum} and $k \times X := \allowbreak \{kx \colon \allowbreak x \in X\}$ for the \evid{$k$-dilate}
of $X$; in particular, $0X = \{0\}$. It is a simple exercise to check that 
\[
n(k \times \mathbb N) = k \times \mathbb N,\qquad \text{for all } n \in \mathbb N^+;
\]
and
\[
nkX := (nk)X = n(kX)
\quad\text{and}\quad
nk \times X := \allowbreak (nk) \times \allowbreak X = n \times (k \times X),\qquad \text{for all } n \in \mathbb N.
\]
Later on, we will often rely on these elementary properties without further comment.

If not explicitly specified, we write all semigroups (and monoids) multiplicatively. An element $a$ in a semigroup $S$ is \evid{cancellative} if $ax \ne ay$ and $xa \ne ya$ for all $x, y \in S$ with $x \ne y$; the semigroup itself is cancellative if each of its elements is. We say on the other hand that $S$ is a \evid{torsion-free} semigroup if the map $\mathbb N^+ \to S \colon x \mapsto
x^n$ is injective for every $x \in S$, except for the identity element of $S$ in the case that $S$ is a monoid. Lastly, we recall that a \evid{unit} of a monoid $M$ with identity element $1_M$ is an element $u \in M$ for which there exists a (provably unique) element $v \in M$, accordingly denoted by $u^{-1}$ and called the \evid{inverse} of $u$ (in $M$), with the property that $uv = vu = 1_M$; and the monoid is \evid{Dedekind-finite} if $xy = 1_M$ for some $x, y \in M$ implies $yx = 1_M$. Commutative monoids and cancellative monoids are Dedekind-finite.

Further notation and ter\-mi\-nol\-o\-gy, if not explained when first used, are standard or should be clear from the context. In particular, we refer to Howie's monograph \cite{Ho95} for basic aspects of semigroup theory.

\section{Groups are globally closed}
\label{sect:2}

In this section, we establish that the class of groups is globally closed (Corollary \ref{cor:groups-are-globally-closed}). In the process, we derive a couple of results (Theorem \ref{prop:global-iso-maps-unit-group-to-unit-group}) that may hold independent interest, 
particularly in relation to Questions \ref{ques:BG-like-for-monoids}. We start with the following:

\begin{definition}
\label{def:unit-stable-element}
Given a monoid $M$, we say that an element $x \in M$ is \evid{unit-stable} if $x = ux = xu$ for all $u \in M^\times$, where $M^\times$ denotes the group of units of $M$. 
\end{definition}

Although every element of a monoid with trivial group of units is unit-stable (which may seem an issue at first glance), Definition \ref{def:unit-stable-element} turns out to be crucial. A few remarks are in order before proceeding.

\begin{remarks}\label{remarks:units-and-unit-stability}
\begin{enumerate*}[label=\textup{(\arabic{*})}, mode=unboxed]
\item\label{remarks:units-and-unit-stability(1)} Let $M$ be a monoid and $X$ be a non-empty subset of $M$. Since $1_M \in M^\times$ and hence $X = \allowbreak X 1_M \subseteq XM^\times$, it is clear that $XM^\times = M^\times$ yields $X \subseteq M^\times$. On the other hand, we have $uM^\times = \allowbreak M^\times$ for every $u \in M^\times$. Therefore, if $X \subseteq M^\times$, then 
\[
XM^\times = \bigcup_{u \in X} uM^\times = M^\times.
\]
By symmetry, this proves that $X \subseteq M^\times$ if and only if $XM^\times = M^\times$, if and only if $M^\times X = M^\times$.
\end{enumerate*}

\vskip 0.05cm

\begin{enumerate*}[label=\textup{(\arabic{*})}, resume, mode=unboxed]
\item\label{remarks:units-and-unit-stability(2)} 
%By \cite{Shaf-1967}, t
The units of the large power monoid $\mathcal{P}(M)$ of a monoid $M$ are precisely 
the singletons $\{u\}$ with $u \in M^\times$ (see Proposition~3.2(ii) in 
\cite{Fa-Tr18} for the analogous statement in the reduced finitary power monoid 
$\mathcal{P}_{\mathrm{fin}}(M)$ of $M$). 
In particular, if $U, V \in \mathcal{P}(M)$ are such that $UV = VU = \{1_M\}$, 
then $uv = vu = 1_M$ for all $u \in U$ and $v \in V$. Therefore, $U$ and $V$ 
are subsets of $M^\times$. This shows that every element of $V$ is cancellative, so that $|U| \le |UV| = 1$ and hence $U = \{u\}$ for some $u \in M^\times$.\\

\indent{}It follows that a non-empty subset $X \subseteq M$ is unit-stable as an element of 
$\mathcal{P}(M)$ if and only if 
\[
uX = Xu = X, \qquad \text{for all } u \in M^\times,
\]
which is equivalent to 
\[
M^\times X = \bigcup_{u \in M^\times} uX = X. 
\]
Conversely, if $M^\times X = X$ and $u \in M^\times$, then 
\[
uX \subseteq M^\times X = X \subseteq u(u^{-1}X) \subseteq u(M^\times X) = uX.
\]
By symmetry, we conclude that $X$ is unit-stable in 
$\mathcal{P}(M)$ if and only if $
M^\times X = XM^\times = X$.
\end{enumerate*}

\vskip 0.05cm

\begin{enumerate*}[label=\textup{(\arabic{*})}, resume, mode=unboxed]
\item\label{remarks:units-and-unit-stability(3)} Let $f$ be a semigroup isomorphism from a monoid $H$ to a monoid $K$. 
It is a basic fact that $f$ sends the identity $1_H$ of $H$ to the identity $1_K$ of $K$ 
(see, e.g., the last lines of \cite[Sect.~2]{Tri-2024(a)}). Hence, $f$ restricts to an isomorphism from $H^\times$ to $K^\times$, 
and it is then routine to check that it preserves unit-stable elements.
%(we leave the simple details to the reader).
%
%\indent{}Indeed, let $x$ be a unit-stable element of $H$ and $v$ be a unit of $K$.  
%We gather from the above and the surjectivity of $f$ that $v = f(u)$ for some unit $u \in H$. So, $vf(x) = \allowbreak f(u) f(x) = \allowbreak f(ux) = f(x)$; and in a similar way, $f(x) = f(x) v$. It follows that $f(x)$ is a unit-stable element of $K$.
\end{enumerate*}
\end{remarks}

We are now in a position to prove the main theorem of the section (and its corollary for groups).

\begin{theorem}
\label{prop:global-iso-maps-unit-group-to-unit-group}
The following hold for a global isomorphism $f$ from a monoid $H$ to a monoid $K$:
\begin{enumerate}[label=\textup{(\roman{*})}]
\item\label{prop:global-iso-maps-unit-group-to-unit-group(1)} 
$f(H^\times) = K^\times$.
\item\label{prop:global-iso-maps-unit-group-to-unit-group(2)} $f$ restricts to a global isomorphism from $H^\times$ to $K^\times$.
\end{enumerate}
\end{theorem}

\begin{proof}
\ref{prop:global-iso-maps-unit-group-to-unit-group(1)} Let $\Omega(M)$ be the set of unit-stable elements of the large power monoid $\mathcal P(M)$ of a monoid $M$. Since the inverse $f^{-1}$ of $f$ is a global isomorphism from $K$ to $H$, it is clear from Remark \ref{remarks:units-and-unit-stability}\ref{remarks:units-and-unit-stability(3)} that $f[\Omega(H)] = \Omega(K)$. On the other hand, $M^\times$ is a unit-stable element of $\mathcal P(M)$, as guaranteed by Remark \ref{remarks:units-and-unit-stability}\ref{remarks:units-and-unit-stability(2)} when noting that $M^\times M^\times = M^\times$.

As a result, $f(H^\times) \in \Omega(K)$, which, by the same Remark \ref{remarks:units-and-unit-stability}\ref{remarks:units-and-unit-stability(3)}, implies $f(H^\times) K^\times = f(H^\times)$. Likewise, $H^\times X = X$, where $X := f^{-1}(K^\times) \in \Omega(H)$. Therefore, 
\[
K^\times = f(X) = f(H^\times) f(X) = f(H^\times) K^\times = f(H^\times).
\]

\ref{prop:global-iso-maps-unit-group-to-unit-group(2)} 
Let $X \in \mathcal P(H^\times)$. By Remark \ref{remarks:units-and-unit-stability}\ref{remarks:units-and-unit-stability(1)}, we have $H^\times = \allowbreak XH^\times$. It thus follows from part \ref{prop:global-iso-maps-unit-group-to-unit-group(1)} that 
\[
K^\times = \allowbreak f(H^\times) = \allowbreak f(X) f(H^\times) = f(X) K^\times, 
\]
which, again by Remark \ref{remarks:units-and-unit-stability}\ref{remarks:units-and-unit-stability(1)}, shows that $f(X) \subseteq \allowbreak K^\times$. This yields $f[\mathcal P(H^\times)] \subseteq \mathcal P(K^\times)$, and the same reasoning applied to $f^{-1}$ establishes the reverse inclusion (thereby completing the proof).
\end{proof}

\begin{corollary}
\label{cor:groups-are-globally-closed}
The class of groups is globally closed.
\end{corollary}

\begin{proof}
Let $f$ be a global isomorphism from a group $H$ to a semigroup $K$. By \cite[Lemma 1.1]{Gou-Isk-Tsi-1984}, $K$ is a monoid. Therefore, we derive from Theorem \ref{prop:global-iso-maps-unit-group-to-unit-group} that $f$ restricts to a global isomorphism from $H^\times$ to $K^\times$. But a group is, by definition, a monoid in which every element is a unit (that is, $H = H^\times)$. So, $f$ is in fact an isomorphism $\mathcal P(H) \to \mathcal P(K^\times)$. Since $\mathcal P(K^\times)$ is contained in $\mathcal P(K)$ and, by hypothesis, $f$ is an iso\-mor\-phism $\mathcal P(H) \to \mathcal P(K)$, this is only possible if $K = K^\times$, which means that $K$ is also a group.
\end{proof}
\section{Torsion-free monoids are globally closed}
\label{sect:3_torsion-free-monoids}
 
For a semigroup $S$, the \evid{order} of an element $x \in S$ is the cardinality of the set $\{x^n : n \in \mathbb{N}^+\}$, that is, the (cyclic) subsemigroup of $S$ generated by $x$. Saying that $S$ is torsion-free is then equivalent to saying that $S$ has no elements of finite order, except for the identity element in the case that $S$ is a monoid. 

Our goal in this short section is to prove that the class of torsion-free monoids is globally closed (Theorem \ref{thm:torsion-free-monoids-are-globally-closed}). We begin with a couple of propositions that may be of independent interest.

\begin{proposition}
\label{prop:S-torsion-free-iff-Pfin(S)}
A semigroup $S$ is torsion-free if and only if its finitary power semigroup $\mathcal P_\fin(S)$ is.
\end{proposition}

\begin{proof}
Since any subsemigroup of a torsion-free sem\-i\-group is obviously torsion-free and the map $S \to \allowbreak \mathcal P_\fin(S) \colon \allowbreak x \mapsto \{x\}$ is a semigroup embedding, it suffices to prove the ``only if'' direction of the claim.

Assume $S$ is torsion-free, and suppose that there exist $m, n \in \allowbreak \mathbb N^+$ with $ m < n$ and a non-empty finite set $X \subseteq S$ such that $X^m = X^n$ in $\mathcal P_\fin(S)$. Then, a routine induction shows that $X^m = \allowbreak X^{m + (n-m)k}$ for every $k \in \mathbb N$, which in turn implies that
$x^{m+(n-m)k} \in X^m$ for all $x \in X$ and $k \in \mathbb N$. 
Since $X^m$ is a finite set, it follows that each $x \in X$ has finite order in $S$. However, this is only possible if $S$ is a monoid and $X$ is the identity element $\{1_S\}$ of $\mathcal P_\fin(S)$. Consequently, $\mathcal P_\fin(S)$ is torsion-free.
\end{proof}

%The next proposition may be of independent interest (cf.~Proposition \ref{prop:nr-of-reps} and Corollary \ref{cor:finite-nr-of-solutions}). 

\begin{proposition}
\label{prop:equation-with-idemps-in-torsion-free-case}
Let $S$ be a non-empty torsion-free semigroup, and suppose that $E$ is an idempotent of the large power semigroup $\mathcal P(S)$ of $S$. Then either $S$ is a monoid with identity $1_S$ and $E$ is the identity $\{1_S\}$ of $\mathcal P(S)$, or the equation $EXE = E$ has infinitely many solutions $X \in \mathcal P(S)$.
\end{proposition}

\begin{proof}
Let $e$ be the identity of a (conditional) unitization of $\widehat{S}$: if $S$ is already a monoid, then $S = \widehat{S}$ and $e = 1_S$; otherwise, $e$ is an element not in $S$ and the multiplication of $S$ is extended to the set $\widehat{S} := S \cup \allowbreak \{e\}$ by requiring that $ex := x =: xe$ for all $x \in \widehat{S}$. 
Accordingly, assume that either $S \ne \widehat{S}$ or $E \ne \{e\}$. We have to prove that the equation $EXE = E$ has infinitely many solutions $X \in \mathcal P(S)$. 

Considering that $S$ is torsion-free and $E$ is an idempotent of $\mathcal P(S)$, it follows from Prop\-o\-si\-tion \ref{prop:S-torsion-free-iff-Pfin(S)} that $|E| = \infty$. We claim that, for each $a \in E \setminus \{e\}$, the (non-empty) set $X_a := \allowbreak E \setminus \{a\} \subseteq S$ is a solution to the equation $EXE = E$. This will clearly complete the proof, as $X_a \ne X_b$ for all distinct $a, b \in E$.

Fix $a \in E \setminus \{e\}$. Since $E = E^n$ for all $n \in \mathbb N^+$ (by the idempotency of $E$), we have $EX_aE \subseteq \allowbreak E^3 = \allowbreak E$. To prove the reverse inclusion, let $x \in E$. Then $x \in E^3$, so $x = uvw$ for some $u, \allowbreak v, \allowbreak w \in \allowbreak E$. If $v \ne a$, then $uvw \in EX_aE$. Otherwise, it follows from $E = E^2$ that $a = bc$ for some $b, c \in E$; moreover, either $a \ne b$ or $a \ne c$, because $a = b = c$ would yield $e \ne a = a^2$ (contradicting that $S$ is torsion-free). If $a \ne \allowbreak b$, then $x = uaw = ub(cw) \in EX_aE^2 = EX_a E$. The case $a \ne c$ is symmetric, and thus we are done.
\end{proof}

We are now ready for the main result of the section. For ease of exposition, we will say that an idem\-potent of a monoid is \evid{non-trivial} if it is not the identity element.

\begin{theorem}
\label{thm:torsion-free-monoids-are-globally-closed}
The class of torsion-free monoids is globally closed.
\end{theorem}

\begin{proof}
Let $f$ be a global isomorphism from a torsion-free monoid $H$ to a semigroup $K$. By \cite[Lemma 1.1]{Gou-Isk-Tsi-1984}, $K$ is itself a monoid. Seeking a contradiction, assume that $K$ is not torsion-free, i.e., $K$ has at least one non-identity element $b$ such that $b^m = b^n$ for some $m, n \in \mathbb N^+$ with $m < n$. It is then clear that $b^k \in \allowbreak \{1_K, \allowbreak b, \ldots, \allowbreak b^{n-1}\}$ for all $k \in \mathbb N$, and hence 
$B := \{1_K, b\}^{n-1} \ne \{1_K\}$
is a non-trivial idempotent of $\mathcal P(K)$.

Since $f$ is an isomorphism from $\mathcal P(H)$ to $\mathcal P(K)$, it follows (in view of Remark \ref{remarks:units-and-unit-stability}\ref{remarks:units-and-unit-stability(3)}) that $A := f^{-1}(B)$ is a non-trivial idempotent of $\mathcal P(H)$. So, by Proposition \ref{prop:equation-with-idemps-in-torsion-free-case} and the torsion-freeness of $H$, the equation $AXA = A$ has infinitely many solutions $X$ in $\mathcal P(H)$. This implies that
$$
B = f(A) = f(AXA) = f(A) f(X) f(A) = B f(X) B
$$
for infinitely many $X \in \mathcal P(H)$, 
which, by the injectivity of $f$, means that the equation $B = BYB$ has infinitely many solutions $Y$ over $\mathcal P(K)$. However, if $BYB = B$ for some $Y \in \mathcal P(K)$, then $Y = \allowbreak 1_K Y 1_K \subseteq \allowbreak BYB = \allowbreak B$, that is, $Y$ is a subset of a finite set (and a finite set has finitely many subsets). We have therefore reached a contradiction, and the proof is complete.
\end{proof}

\begin{comment}
\begin{remark}
A cancellative monoid need not be torsion-free. For instance, finite groups are cancellative but not torsion-free.
\end{remark}

\begin{lemma}
If $f$ is a global isomorphism from a cancellative monoid $H$ to a monoid $K$, then every non-unit of $K$ has infinite order.
\end{lemma}

\begin{proof}
Suppose to the contrary that $K$ has a non-unit $y$ of finite order. We may assume that $y = y^2$; otherwise, we can find a positive integer $k$ such that $y^k = y^{2k}$ and replace $y$ with $y^k$, upon considering that $y^k$ is a non-unit (or else $y$ would be a unit). Set $X := f^{-1}(\{1_K, y\})$. Since $\{1_K\} \ne \{1_K, y\} = \{1_K, y\}^2$, $X$ must be a non-identity idempotent of $\mathcal P(H)$. 

We claim $X \subseteq H^\times$. In fact, assume to the contrary that $X$ contains a non-unit $x$. Then $x^n \in X^n = X$ for all $n \in \mathbb N^+$ (because $X = X^2$ and hence $X = X^n$). However, every non-unit of $H$ has infinite order, by the fact that $H$ is cancellative and every finite-order elements of a cancellative monoid is a unit (even if not every unit need have finite order). It follows that $x, x^2, \ldots$ are pairwise distinct elements in $X$, with the result that $|X| = \infty$ (which is a contradiction). 

So, every element of $X$ is a unit. But we proved in Theorem 2.3(i) of our paper that $f[\mathcal P(H^\times)] = \mathcal P(K^\times)$. Hence, $\{1_K, y\} = f(X) \in \mathcal P(K^\times)$, which is impossible since $y$ is a non-unit of $K$.
\end{proof}
\end{comment}
\section{An interlude in additive number theory}
\label{sec:additive_NT}

The present section contains several results of a combinatorial nature that serve as preliminaries for the proof of Theorem \ref{thm:numerical-monoids-are-globally-closed}. We begin with a couple of elementary lemmas that are certainly well known; as we have not been able to find a reference, we include their proofs here for completeness.

\begin{lemma}
\label{lem:n-fold-sum-of-AP}
Let $k$ be a positive integer and $A$ be a subset of $\mathbb N$ containing $0$. If $X := A + \dil{k}{\mathbb N}$, then $nX = (n+1)X$ for all large $n \in \mathbb N$.
\end{lemma}

\begin{proof}
Given $n \in \mathbb N^+$, let $R_n$ be the set of residues $r \in \llb 0, k-1 \rrb$ such that $a \equiv r \bmod k$ for some $a \in \allowbreak nA$; and for each $r \in R_n$, let $a_{n,r}$ be the smallest element in $nA$ congruent to $r$ modulo $k$. 
Considering that $n(\dil{k}{\mathbb N}) = \dil{k}{\mathbb N}$ and taking $A_n := \allowbreak \{a_{n,r} \colon r \in R_n\}$, it is readily seen that
\begin{equation}
\label{lem:n-fold-sum-of-AP:equ(1)}
nX = nA + n (\dil{k}{\mathbb N}) = A_n + k \times \mathbb N.
\end{equation}
On the other hand, $0 \in A$ (by hypothesis) yields $nA \subseteq (n+1)A$ and hence $R_n \subseteq R_
{n+1} \subseteq \llb 0, k-1 \rrb$. It follows that (i) $0 \le a_{n+1,r} \le a_{n,r}$ for every $r \in R_n$ and (ii) $R_n = R_{n+1}$ for all but finitely many values of $n$.
Since there is no strictly descending (infinite) sequence of non-negative integers, we conclude that $A_n = \allowbreak A_{n+1}$ for any sufficiently large $n$, which implies by Eq.~\eqref{lem:n-fold-sum-of-AP:equ(1)} that 
$
nX = A_{n+1} + k \times \allowbreak \mathbb N = (n+1) X$.
\end{proof}

\begin{lemma}
\label{lem:torsion-free-implies-aperiodic}
Let $X$ be a subset of $\mathbb N$, and suppose that $hX = kX$ for some $h, k \in \mathbb N$ with $h \ne k$. Then $nX = (n+1)X$ for every large $n \in \mathbb N$.
\end{lemma}

\begin{proof}
Assume without loss of generality that $h < k$. If $X = \emptyset$, then $nX = \emptyset$ for all $n \in \mathbb N$ (and we are done). Otherwise, $hX = kX$ yields $h \min X = k \min X$ and hence $0 = \min X \in X$. It is then clear that 
\[
hX = hX + 0 \subseteq hX + X \subseteq hX + (k-h)X = kX = hX.
\]
This leads to $hX = (h+1)X$, which, by a routine induction, implies $nX = (n+1)X$ for any $n \ge h$.
\end{proof}

Among other things, we will make use of a classical theorem of Kneser \cite[Theorem I.17$^\prime$]{Halb-Roth-1967} on the ``structure'' of a bounded-from-below set $X$ of integers whose \evid{lower asymptotic density}
\begin{equation}
\label{equ:def-of-lower-asymptotic-density}
\mathsf{d}_\ast(X) := \liminf_{n \to \infty} \frac{\bigl|X \cap \llb 1, n \rrb\bigr|}{n}
\end{equation}
is not ``too small''. (Note how the limit in Eq.~\eqref{equ:def-of-lower-asymptotic-density} only takes into account the \textit{positive} elements of $X$, consistently with Halberstam and Roth's definition from loc.~cit., p.~xvii).
For the reader's convenience, we restate Kneser's theorem here in a form that is best suited to our purposes.

\begin{theorem}[Kneser's theorem]
\label{thm:kneser-thm}
Let $X$ be a subset of $\mathbb N$. If $\mathrm{d}_*(2X) < 2\mathrm{d}_*(X)$, then there exist an integer $m \ge 1$ and a set $A \subseteq \llb 0, m-1 \rrb$ such that $X$ is contained in the set $Y := A + \dil{m}{\mathbb N}$ and the difference $2Y \setminus 2X$ is finite.
\end{theorem}

In fact, we will apply Kneser's theorem right away to prove Proposition \ref{prop:nX=2nX} below, which is in turn a key ingredient in the proof of Theorem \ref{thm:numerical-monoids-are-globally-closed}. But first, we need another lemma.

\begin{lemma}
\label{lem:kneser-thm-application}
If $0 \in X \subseteq \mathbb N$ and $\mathrm{d}_\ast(2X) < 2 \mathrm{d}_\ast(X)$, then $nX = (n+1)X$ for all large $n \in \mathbb N$.
\end{lemma}

\begin{proof}
By Theorem \ref{thm:kneser-thm} (applied to $X$), there exist $m \in \mathbb N^+$ and $A \subseteq \llb 0, m-1 \rrb$ such that 
\begin{equation}
\label{lem:kneser-thm-application:equ(01)}
X \subseteq Y := A + m \times \mathbb N
\end{equation}
and
\begin{equation}
\label{lem:kneser-thm-application:equ(02)}
2X \cap \llb k, \infty \rrb = 2Y \cap \llb k, \infty \rrb,
\qquad \text{for some } k \in \mathbb N.
\end{equation}
By the hypothesis that $0 \in X$, Eq.~\eqref{lem:kneser-thm-application:equ(01)} 
implies $0 \in A$ and $0 \in 2X \subseteq 2Y$. Hence, by Eq.~\eqref{lem:kneser-thm-application:equ(02)},
\begin{equation}
\label{lem:kneser-thm-application:equ(03)}
2X = (2Y) \setminus Q,
\qquad\text{for a certain }
Q \subseteq \allowbreak \llb 1, k - \allowbreak 1 \rrb. 
\end{equation}
On the other hand, it is straightforward that $\mathbb N = k \times \mathbb N + \llb 0, k-1 \rrb$; in addition, $m \times (S + \allowbreak T) = m \times \allowbreak S + \allowbreak m \times T$ for all $S, T \subseteq \mathbb N$. So, setting $B := 2A + m \times \llb 0, k-1 \rrb$, we obtain
\[
2Y = 2A + 2(m \times \mathbb N) = 2A + m \times \mathbb N = (2A + m \times \llb 0, k-1 \rrb) + m \times (k \times \mathbb N) = B + mk \times \mathbb N.
\]
Since $Q \subseteq \llb 1, k-1 \rrb$ and $\sup Q < k \le mk$ (with $\sup \emptyset := 0$), it is then clear that $Q \subseteq B$. It follows, by Eq.~\eqref{lem:kneser-thm-application:equ(03)} and a standard double inclusion, that 
\[
2X = (B + mk \times \mathbb N) \setminus Q = C + mk \times \mathbb N,
\qquad\text{where }
C := (B \setminus Q) \cup (B + mk \times \mathbb N^+). 
\]
Given that $0 \in A \subseteq B$ and hence $0 \in C$, we thus infer from Lemma \ref{lem:n-fold-sum-of-AP} (applied to the set $2X$) that $2nX = \allowbreak 2(n+1)X$ for every large $n \in \mathbb N$, which, by Lemma \ref{lem:torsion-free-implies-aperiodic}, finishes the proof.
\begin{comment}
\begin{equation}
\label{lem:kneser-thm-application:equ(03)}
%nX \subseteq (n+1)X
%\qquad\text{and}\qquad
nY = (n+1)Y,
\qquad\text{for every sufficiently large }
n \in \mathbb N.
\end{equation}
Since, on the other hand, $hX \subseteq (h+1)X$ for all $h \in \mathbb N$ (by the fact that $0 \in X$), we find that, for all large enough $h \in \mathbb N^+$ (at least as large as $n$),
\[
nY \cap \llb m, \infty \rrb \subseteq hX \subseteq (h+1)X \subseteq (h+1)Y = hY = nY.
\]
(we know that $nY = (n+1)Y$ and hence $hY = nY$ for all large $h \in \mathbb N$, this explain the last equality.) This implies that the sets $X, 2X, \ldots$ can only differ, from some point, by finitely many elements in the interval $\llb 0, m-1 \rrb$. Since the sequence $X, 2X, \ldots$ is monotone non-decreasing with respect to the inclusion order, this is only possible if there exists $h_0 \in \mathbb N^+$ such that $(h+1)X = hX$ for every integer $h \ge h_0$, which finishes the proof. 
\end{comment}
\end{proof}

\begin{comment}
\begin{lemma}
\label{dinf-doubling}
If $X$ is a subset of $\mathbb N$ such that $\dinf(X) > 0$, then $\dinf(2^{r+1}X) <  2\dinf(2^{r}X)$ for some $r \in \mathbb N$.
\end{lemma}

\begin{proof}
Assume for a contradiction that $\dinf(2^{r+1} X) \ge 2\dinf(2^r X)$ for all $r \in \mathbb N$. Then, a routine induction shows that $\dinf(2^r X) \ge 2^r \dinf(X)$ for every $r \in \mathbb N$. Since $\dinf(X) > 0$, this implies that $\dinf(2^r X) > 1$ for a sufficiently large $r \in \mathbb N$, which is however impossible because $0 \le \dinf(Y) \le 1$ for any set $Y \subseteq \mathbb N$.
\end{proof}
\end{comment}

\begin{proposition}
\label{prop:nX=2nX}
Let $X$ be a subset of $\mathbb N$ containing $0$, and suppose that 
\begin{equation}
\label{lem:nX=2nX:equ(01)}
X + \{0, u\} = X + \{0, v\}
\end{equation}
for some $u, v \in \mathbb N$ with $u \ne v$. Then $nX = (n+1)X$ for every large $n \in \mathbb N$.
\end{proposition}

\begin{proof}
Assume without loss of generality that $u < v$, and observe that $X$ is an infinite set. Otherwise, $X$ would have a maximum element, and we would obtain from Eq.~\eqref{lem:nX=2nX:equ(01)} that
\[
\max X + u = \max(X + \{0, u\}) = \max(X + \{0, v\}) = \max X + v, 
\]
which is absurd, as it implies $u = v$. Accordingly, let $x_0, x_1, x_2, \ldots$ be the natural enumeration of $X$, so that $X = \{x_0, x_1, x_2, \ldots\}$ and $0 = x_0 \le x_k < x_{k+1}$ for all $k \in \mathbb N$. Next, define
\[
\Delta(X) := \{x_{k+1} - x_k : k \in \mathbb N\} \subseteq \mathbb N^+,
\]
and suppose towards a contradiction that $\Delta(X)$ is an infinite set. There then exists $\kappa \in \mathbb N$ such that 
\begin{equation}\label{lem:nX=2nX:equ(02)}
d := x_{\kappa+1} - x_\kappa \ge 1 + v.
\end{equation}
On the other hand, Eq.~\eqref{lem:nX=2nX:equ(01)} yields $x_\kappa + v \in X + \{0, u\}$ and hence 
$x_\kappa + \allowbreak v \in x + \{0, u\}$ for some $x \in X$.
This is however impossible. In fact, either $x > x_\kappa$, and then, by Eq.~\eqref{lem:nX=2nX:equ(02)}, $x+u \ge x \ge x_{\kappa + 1} = \allowbreak x_\kappa + \allowbreak v$; or $x \le x_\kappa$, and then $x \le x + u < x_\kappa + v$ (recall that we are assuming $u < v$). 

It follows that $\Delta(X)$ is a non-empty finite subset of $\mathbb{N}^+$ and therefore has a maximum element. It is then straightforward (by induction) to verify that 
\[
x_k \le x_0 + k \max \Delta(X) = k \max \Delta(X),
\qquad\text{for every }
k \in \mathbb N.
\]
Thus, the lower asymptotic density $\mathrm{d}_\ast(X)$ of $X$ is positive, as we get from \cite[Theorem 11.1]{Nive-Zuck-Mont-1991} that
\begin{equation}\label{lem:nX=2nX:equ(03)}
\mathrm{d}_\ast(X) 
= \liminf_{k \to \infty} \frac{k}{x_k} \ge \liminf_{k \to \infty} \frac{k}{k \max \Delta(X)} = \frac{1}{\max \Delta(X)} > 0.
\end{equation}
Now, suppose for a contradiction that 
$\dinf(2^{r+1} X) \ge 2\dinf(2^r X)$ for all $
r \in \mathbb N$.
A simple induction then shows that $\dinf(2^r X) \ge 2^r \dinf(X)$ for each $r \in \mathbb N$, which, together with Eq.~\eqref{lem:nX=2nX:equ(03)}, implies $\dinf(2^r X) > \allowbreak 1$ for $r$ (strictly) greater than $-\log_2 \mathsf{d}_\ast(X)$. This is however absurd, as $0 \le \dinf(Y) \le 1$ for any $Y \subseteq \mathbb N$.

Consequently, we are guaranteed that there exists an integer $r \ge 0$ such that $\mathrm{d}_\ast(2^{r+1} X) < 2 \mathrm{d}_\ast(2^r X)$; and applying Lemma \ref{lem:kneser-thm-application} to the set $X' := 2^r X$, we conclude that $2^r mX = mX' = 2mX' = 2^{r+1}mX$ for some $m \in \allowbreak \mathbb N^+$, which, by Lemma \ref{lem:torsion-free-implies-aperiodic}, completes the proof.
\end{proof}

We end the section with a folklore result that will play a role in the proof of Lemma \ref{lem:nearly-cancellative}.

\begin{proposition}
\label{prop:folklore-inequality}
If $X$ and $Y$ are non-empty subsets of $\mathbb Z$, then $|X+Y| \ge |X| + |Y| - 1$.
\end{proposition}

\begin{proof}
This is, for instance, a special case of \cite[Proposition 3.5(i)]{Fa-Tr18}.
\end{proof}
\section{Numerical monoids are globally closed}
\label{sect:numerical-monoids-are-gloobally-closed}

Below, we focus on proving that the class of numerical monoids is globally closed (Theorem \ref{thm:numerical-monoids-are-globally-closed}). We start with a proposition and a couple of corollaries that might be of independent interest.

\begin{proposition}
\label{prop:nr-of-reps}
Let $M$ be a monoid, $S$ be a subset of $M$ containing the identity $1_M$, and $n$ be an integer $\ge 3$. If $1_M \notin T \subseteq S$, then $
(S^{n-1} \setminus T) S = S^n$. In particular, the equation $AS = S^n$ has at least $2^{|S|-1}$ solutions $A$ over the large power monoid $\mathcal P(M)$ of $M$.
\end{proposition}

\begin{proof}
Let $T$ be a subset of $S \setminus \{1_M\}$, and define $Q := S^{n-1} \setminus T$. It is obvious that $QS \subseteq \allowbreak S^{n-1} S = \allowbreak S^n$, so we only need to check that $S^n \subseteq QS$. 
To this end, fix $z \in S^n$, and let $k$ be the smallest integer $\ge 0$ such that $z \in S^k$ (it is clear that $0 \le k \le n$; and since $1_M \in S$, we have $S \subseteq S^2 \subseteq \cdots \subseteq S^n$). Our goal is to show that $z \in QS$, and we distinguish three cases.

\begin{itemize}[leftmargin=0.3in]
\item\textsc{Case 1:} $k = 0$ or $k = 1$. If $k = 0$, then $z \in S^0 = \{1_M\}$ and thus $z = 1_M$. It follows that, regardless of whether $k = 0$ or $k = 1$, $z \in S$ and hence $z \in 1_M S \subseteq QS$ (note that $1_M \in Q$).

\vskip 0.05cm

\item\textsc{Case 2:} $2 \le k \le n-1$. We have $z \notin T$, or else $z \in S$ (by the fact that $T \subseteq S$), contradicting the definition itself of $k$ (which guarantees that $z \notin S^i$ for every non-negative integer $i < k$). It follows that $z \in S^k \setminus T \subseteq Q$ (by the fact that $S^k \subseteq S^{n-1}$), and hence $z \in Q1_M \subseteq QS$.

\vskip 0.05cm

\item\textsc{Case 3:} $k = n$. By the definition of $k$, we have that $z \notin S^i$ for each $i \in \llb 1, n-1 \rrb$. On the other hand, $z = s_1 \cdots s_n$ for some $s_1, \ldots, s_n \in S$. It follows that $z' := s_1 \cdots s_{n-1} \notin T$, or else $z = z' s_n \in S^2$, which is a contradiction because $2 < 3 \le n$. Therefore, $z' \in S^{n-1} \setminus T$ and hence $z = z's_n \in Q S$.
\end{itemize}

As for the ``In particular'' part of the statement, we have already observed that $S \subseteq S^{n-1}$. So, if $T_1$ and $T_2$ are distinct subsets of $S \setminus \{1_M\}$, then the sets $S^{n-1} \setminus T_1$ and $S^{n-1} \setminus T_2$ are likewise distinct. As a result, the equation $AS = S^n$ has at least as many solutions $A$ over $\mathcal P(M)$ as there are subsets of $S \setminus \{1_M\}$; namely, it has at least $2^{|S|-1}$ solutions.
\end{proof}

\begin{corollary}
\label{cor:finite-nr-of-solutions}
Let $M$ be a monoid and $X$ be a subset of $M$ containing the identity $1_M$. Then the equation $AX = X^3$ has finitely many solutions $A$ in $\mathcal P(M)$ if and only if $X$ is a finite set.
\end{corollary}

\begin{proof}
If $AX = X^3$ for some $A \in \mathcal P(M)$, then $1_M \in X$ implies $A \subseteq \allowbreak AX = \allowbreak X^3$. Since $|X^3| \le \allowbreak |X|^3$, it follows that if $X$ is finite, then there exist finitely many $A \in \mathcal P(M)$ such that $AX = X^3$.
On the other hand, we have from Proposition \ref{prop:nr-of-reps} (applied with $S = X$ and $n = 3$) that if $X$ is an infinite set, then the equation $AX = X^3$ has infinitely many solutions $A$ in $\mathcal P(M)$.
\end{proof}

For the next corollaries, 
we recall from Sect.~\ref{sect:1} that $\mathcal{P}_1(M)$ (resp., $\mathcal P_{\fin,1}(M)$) denotes the submonoid of the large power monoid $\mathcal{P}(M)$ of $M$ consisting 
of all subsets (resp., finite subsets) of $M$ that contain the identity. 
If $M$ is in particular a numerical monoid, we write $\mathcal P(M)$ and its submonoids additively, in contrast to what we have done so far for arbitrary semigroups. Accordingly, we adjust the notation to fit the context, using $\mathcal P_0(M)$ instead of $\mathcal P_1(M)$ and $\mathcal P_{\fin,0}(M)$ instead of $\mathcal P_{\fin,1}(M)$.

\begin{corollary}
\label{cor:finite-back-to-finite}
Let $g$ be a global isomorphism from a monoid $K$ to a Dedekind-finite monoid $H$ with trivial group of units. If $g(K) H = H$, then
\begin{enumerate*}[label=\textup{(\roman{*})}]
\item\label{cor:finite-back-to-finite(i)} $g[\mathcal P_1(K)] \subseteq \mathcal P_1(H)$ and
\item\label{cor:finite-back-to-finite(ii)} $g[\mathcal P_{\fin,1}(K)] \subseteq \mathcal P_{\fin,1}(H)$.
\end{enumerate*} 
\end{corollary}

\begin{proof}
\ref{cor:finite-back-to-finite(i)} If $X \in \mathcal P_1(K)$, then $K = 1_K K \subseteq XK \subseteq K$, namely, $XK = K$. It follows that $g(X) g(K) = g(K)$. Under the assumption that $g(K) H = H$, this implies $
g(X) H = \allowbreak g(X) g(K) H = \allowbreak g(K) H = \allowbreak H$, and hence $1_H = xy$ for some $x \in g(X)$ and $y \in H$. Since $H$ is a Dedekind-finite monoid and its group of units is trivial (by hypothesis), we can therefore conclude that $x = 1_H$. Consequently, $g(X) \in \mathcal P_1(H)$.

\vskip 0.05cm

\ref{cor:finite-back-to-finite(ii)} Fix $X \in \mathcal P_{\fin,1}(K)$. It will be enough to prove that $Y := g(X) \subseteq H$ is a finite set, as we already know from part \ref{cor:finite-back-to-finite(i)} that $1_H \in Y$. Given $T \in \mathcal P(M)$, denote by $\mathcal C_M(T)$ the family of all $S \in \allowbreak \mathcal P(M)$ such that $ST = \allowbreak T^3$, where $M$ is either $H$ or $K$. 
It is clear that $A \in \allowbreak \mathcal{C}_K(X)$ yields $g(A) Y = Y^3$ and hence $g(A) \in \allowbreak \mathcal C_H(Y)$. Conversely, $B \in \mathcal{C}_H(Y)$ yields $g^{-1}(B) X = X^3$ and hence $g^{-1}(B) \in \allowbreak \mathcal C_K(X)$. 

Thus, $g$ establishes a bijection from $\mathcal{C}_K(X)$ to $\mathcal{C}_H(Y)$, with the result that $|\mathcal C_K(X)| = |\mathcal C_H(Y)|$. 
On the other hand, we are guaranteed by Corollary \ref{cor:finite-nr-of-solutions} that $\mathcal C_K(X)$ is a finite set (here we use that $1_K \in X$). It follows that $\mathcal C_H(Y)$ is finite too, which, by another application of Corollary \ref{cor:finite-nr-of-solutions} (here we use that $1_H \in Y$), is only possible if $|Y| < \infty$ (as desired).
\end{proof}

\begin{corollary}
\label{cor:4.5}
If $g$ is a global isomorphism from a monoid $K$ to a numerical monoid $H$, then $g[\mathcal P_1(K)] \subseteq \mathcal P_0(H)$ and $g[\mathcal P_{\fin,1}(K)] \subseteq \mathcal P_{\fin,0}(H)$.
\end{corollary}

\begin{proof}
If $E$ is an i\-dem\-po\-tent of $\mathcal P(H)$, then $2\min E = \min E$ and hence $0 = \min E \in E$. On the other hand, $K = 1_K K \subseteq K^2 \subseteq K$; that is, $K$ is an idempotent of $\mathcal P(K)$. Since a semigroup isomorphism maps idempotents to idempotents, it is thus clear that $0 \in g(K)$. It follows that $H \subseteq g(K) + \allowbreak H \subseteq \allowbreak H$, namely, $g(K) + H = \allowbreak H$. By Corollary \ref{cor:finite-back-to-finite}, this suffices to finish the proof, for $H$ is a commutative monoid with trivial group of units, and commutative monoids are Dedekind-finite.
\end{proof}

\begin{remark}
\label{remark:conjecture}
We conjecture that every global isomorphism $g$ from a monoid $K$ to a Dedekind-finite monoid $H$ satisfies $g(K) = H$. If true, this would imply (see Sect.~\ref{sect:1} for notation) that
\[
g[\mathcal P_\times(K)] = \mathcal P_\times(H)
\qquad\text{and}\qquad
g[\mathcal P_{\fin,\times}(K)] = \mathcal P_{\fin,\times}(H),
\]
thereby making Corollary \ref{cor:finite-back-to-finite} much smoother. For sure, nothing similar holds for arbitrary semigroups. 

For instance, let $S$ be a \evid{left zero semigroup}, that is, a semigroup with the property that $xy = x$ for all $x, y \in S$. Then $XY = X$ for all $X, Y \in \mathcal P(S)$, with the result that any permutation of $\mathcal P(S)$ is a global isomorphism of $S$. It follows that if $|S| \ge 2$, then there is an automorphism $f$ of $\mathcal P(S)$ such that $f(S) \ne \allowbreak S$ (fix $x \in S$ and let $f$ be the transposition of $\mathcal P(S)$ that swaps $\{x\}$ and $S$).
\end{remark}

%Proving Theorem \ref{thm:numerical-monoids-are-globally-closed} requires, in fact, further preparations. We continue with a special case of Cor\-ol\-lar\-y \ref{cor:finite-back-to-finite} and then 
We break down the remainder of the proof into a series of lemmas.

\begin{comment}
\begin{lemma}
\label{lem:torsion-free}
If a numerical monoid $H$ is globally isomorphic to a monoid $K$, then $K$ is torsion-free. 
\end{lemma}

\begin{proof}
This is straightforward from Theorem \ref{thm:torsion-free-monoids-are-globally-closed}.
Let $f$ be a global isomorphism from $H$ to $K$, and let $y$ be a finite-order element of $K$. There then exist $m, n \in \mathbb N^+$ with $m < n$ such that $y^m = y^n$, implying that
%
\begin{equation}
\label{lem:torsion-free:equ_(01)}
\{1_K, y\}^n = \{1_K, y, \ldots, y^m, \ldots, y^{n-1}, y^n\} = \{1_K, y, \ldots, y^{n-1}\} = \{1_K, y\}^{n-1}.
\end{equation}
%
We need to prove that $y = 1_K$. Set $X := f^{-1}(\{1_K, y\})$. 
By Corollary \ref{cor:4.5} (applied to $f^{-1}$), $X$ is a finite subset of $H$ containing $0$.
So, it follows from Eq.~\eqref{lem:torsion-free:equ_(01)} that $nX = (n-1)X$, which yields $n \max X = \allowbreak (n - \allowbreak 1)\max X$ and hence $\max X = 0$, that is, $X = \{0\}$. This in turn leads to $\{1_K, y\} = \allowbreak f(X) = f(\{0\}) = \allowbreak \{1_K\}$, namely, $y = 1_K$ (as desired).
\end{proof}
\end{comment}

\begin{lemma}
\label{lem:2-to-2-for-numerical-monoids}
Let $f$ be a global isomorphism from a numerical monoid $H$ to a monoid $K$, and fix a non-identity element $y \in K$. Then $f^{-1}(\{1_K, y\}) = \{0, x\}$ for some $x \in H$.
\end{lemma}

\begin{proof}
Set $Y := \{1_K, y\}$. 
From Corollary \ref{cor:4.5} (applied to $f^{-1}$), we have that $X := f^{-1}(Y)$ is a finite subset of $H$ containing $0$. 
Since $f$ is injective and maps the identity $\{0\}$ of $\mathcal P(H)$ to the identity $\{1_K\}$ of $\mathcal P(K)$ (see Remark \ref{remarks:units-and-unit-stability}\ref{remarks:units-and-unit-stability(3)}), it is then obvious that $k := |X| - 1 \in \mathbb N^+$. The claim now reduces to proving that $k = 1$. To this end, let us consider the families
$$
\mathcal{C}_H(X) := \{A \in \mathcal P(H) \colon A + X = 3X\}
\qquad\text{and}\qquad
\mathcal{C}_K(Y) := \{B \in \mathcal P(K) \colon BY = Y^3\}.
$$
Similarly as in the proof of Corollary \ref{cor:finite-back-to-finite}\ref{cor:finite-back-to-finite(ii)}, we have $d := |\mathcal{C}_H(X)| = |\mathcal{C}_K(Y)|$. In particular, we gather
from the above and Proposition \ref{prop:nr-of-reps} (applied with $S = X$ and $n = 3$) that $d = |\mathcal{C}_H(X)| \ge 2^k \ge 2$. 

On the other hand, Theorem \ref{thm:torsion-free-monoids-are-globally-closed} guarantees that $K$ is torsion-free, as it is globally isomorphic to a numerical monoid and numerical monoids are torsion-free. Hence, $y^4 \notin Y^3$ and $y^3 \notin Y^2$. So, if $B \in \allowbreak \mathcal{C}_K(Y)$, then $y^3 \notin B$, and thus $B \subseteq Y^2$. Moreover, $B \ne \{1_K\}$ (otherwise $BY = Y \ne Y^3)$ and $y^2 \in B$ (otherwise $B \subseteq Y$, and hence $BY \subseteq Y^2 \subsetneq Y^3$). It follows that $B \in \mathcal{C}_K(Y)$ if and only if $B = \{1_K, y^2\}$ or $B = Y^2$, with the result that $2 \le 2^k \le d = |\mathcal{C}_K(Y)| = 2$. This is only possible if $k = 1$ (as desired).
\end{proof}

\begin{lemma}
\label{lem:nearly-cancellative}
If a numerical monoid $H$ is globally isomorphic to a monoid $K$, then $x^2 \ne xy$ for all $x, y \in K$ with $x \ne y$.
\end{lemma}

\begin{proof}
Let $f$ be a global isomorphism from $H$ to $K$, and suppose for a contradiction that $x^2 = xy$ for some $x, y \in K$ with $x \ne y$. Accordingly, we have
\begin{equation}
\label{lem:nearly-cancellative:equ_01}
\{1_K, x\}\{1_K, y\} = \{1_K, x, y, xy\} = \{1_K, x, y, xy, x^2\} = \{1_K, x\} \{1_K, x, y\},
\end{equation}
all calculations being carried over in the large power monoid $\mathcal P(K)$ of $K$. Set 
\[
X := f^{-1}(\{1_K, x\}),
\qquad 
Y := f^{-1}(\{1_K, y\}),
\qquad\text{and}\qquad
Z := f^{-1}(\{1_K, x, y\}).
\]
It is clear that $x \ne 1_K$, or else $y = xy = x^2 = 1_K = x$ (a contradiction). In turn, this yields $y \ne 1_K$; otherwise, $x^2 = xy = x \ne 1_K$, which is impossible, as $H$ is torsion-free and, by Theorem~\ref{thm:torsion-free-monoids-are-globally-closed}, so must be $K$.
In view of Lemma \ref{lem:2-to-2-for-numerical-monoids}, it follows that $X = \{0, a\}$ and $Y = \{0, b\}$ for some non-zero $a, b \in H$ with $a \ne b$, and we may assume without loss of generality that $a < b$. Moreover, Corollary \ref{cor:finite-back-to-finite} shows that $\{0\} \ne \allowbreak Z \in \allowbreak \mathcal P_{\fin,0}(H)$, and hence $n := \allowbreak |Z| - 1 \in \mathbb N^+$. Consequently, we gather from Eq.~\eqref{lem:nearly-cancellative:equ_01} that
\begin{equation}
\label{lem:nearly-cancellative:equ_02}
\{0, a, b, a+b\} = X + Y = X + Z. 
\end{equation}
Since $1 \le a < b < a+b$, it is now immediate from the above and Proposition \ref{prop:folklore-inequality} that 
\[
4 = |X+Y| = |X + Z| \ge |X| + |Z| - 1 = n+2, 
\]
Thus, $1 \le n \le 2$, and we claim that $n = 2$. Suppose to the contrary that $n = 1$, namely, $Z = \{0, c\}$ for some non-zero $c \in H$. It is then straightforward from Eq.~\eqref{lem:nearly-cancellative:equ_02} that 
\[
a + b = \max X + \max Y = \max X + \max Z = a + c
\]
which implies $Y = Z$ and hence contradicts the fact that $f(Y) = \{1_K, y\} \ne \{1_K, x, y\} = f(Z)$.

All in all, we have therefore established that $Z = \{0, u, v\}$ for some non-zero $u, v \in H$ with $u < v$. Together with Eq.~\eqref{lem:nearly-cancellative:equ_02}, this leads us to conclude that
\[
\{0, a, u, v, a+u, a+v\} = \{0, a\} + \{0, u, v\} = \{0, a, b, a+b\}.
\]
It is then a simple exercise to check that $u = a$ and hence $v = b = 2a$ (we leave any further details to the reader). Consequently, we obtain $Z = 2X$, from which
\[
\{1_K, x, y\} = f(Z) = f(2X) = \{1_K, x\}^2 = \{1_H, x, x^2\}.
\]
It follows that $1_K \ne y = x^2$ and hence $1_K \ne x^2 = xy = x^3$. This is however impossible (and finishes the proof), as we have already noted that $K$ is a torsion-free monoid.
\end{proof}

\begin{lemma}
\label{lem:numerica-monoid-only-gl-iso-to-cancellative}
If a numerical monoid $H$ is globally isomorphic to a commutative monoid $K$, then $K$ is cancellative.
\end{lemma}

\begin{proof}
Let $f$ be a global isomorphism from $H$ to $K$, and suppose towards a contradiction that $K$ is not cancellative, that is, there exist $x, y, z \in K$ with $y \ne z$ and $xy = xz$. Then, in $\mathcal P(K)$, we have
\begin{equation}
\label{eq:lem:numerica-monoid-only-gl-iso-to-cancellative(1)}
x \, \{1_K, y\} = x \, \{1_K, z\}.
\end{equation}
Set $X := f^{-1}(\{x\})$ and $X' := X - \min X$. Considering that $0 \in X' \subseteq \mathbb N$, we gather from Eq.~\eqref{eq:lem:numerica-monoid-only-gl-iso-to-cancellative(1)}
%, Corollary \ref{cor:finite-back-to-finite}, 
and Lemma \ref{lem:2-to-2-for-numerical-monoids} that $X' +\{0, u\} = X' + \{0, v\} $ for some $u, v \in H$ with $u \ne v$. Thus, Proposition \ref{prop:nX=2nX} guarantees that $nX' = 2nX'$ for a certain $n \in \mathbb N^+$, which in turn implies
\begin{equation}\label{eq:lem:numerica-monoid-only-gl-iso-to-cancellative(2)}
2nX = 2nX' + 2n \min X = nX' + 2n \min X = nX + n\min X.
\end{equation}
Now, by \cite[Proposition 1]{Tri-2024(a)}, every cancellative element of the large power sem\-i\-group of a \textit{commutative} semigroup $S$ is a singleton containing a cancellative element of $S$. Since $H$ is a cancellative monoid, $\min X$ is an element of $H$, and a semigroup isomorphism maps cancellative elements to cancellative elements \cite[Remark 2]{Tri-2024(a)}, it follows that $f(\{\min X\}) = \{a\}$ for a cancellative element $a \in K$. Therefore, from Eq.~\eqref{eq:lem:numerica-monoid-only-gl-iso-to-cancellative(2)}, we obtain that $x$ and $a$ satisfy the relation $x^{2n} = x^n a^n$ in $K$. By Lemma \ref{lem:nearly-cancellative}, however, this is only possible if $x^n = a^n$, which leads to a contradiction and completes the proof, because $a^n$ is a cancellative element of $K$, but $x^n$ is not. (It is a basic exercise to show that, in a semigroup, an element is cancellative if and only if all its powers are.)
\end{proof}

Finally, we have all the necessary ingredients to prove the main result of this section.

\begin{theorem}
\label{thm:numerical-monoids-are-globally-closed}
The class of numerical monoids is globally closed.
\end{theorem}

\begin{proof}
Let $f$ be a global isomorphism from a numerical monoid $H$ to a semigroup $K$. By \cite[Remark 3]{Tri-2024(a)} and \cite[Lemma 1.1]{Gou-Isk-Tsi-1984}, $K$ is a fortiori a commutative monoid. It thus follows from Lemma \ref{lem:numerica-monoid-only-gl-iso-to-cancellative} that $K$ is also cancellative. Accordingly, we conclude from \cite[Corollary 1]{Tri-2024(a)} that $H$ is isomorphic to $K$, which ultimately proves that the class of numerical monoids is globally closed.
\end{proof}
\section*{Acknowledgments}
This work was supported by the Natural Science Foundation of Hebei Province under grant A2023205045.
%The second author dedicates this work to his friends, Pietro and Carmelo Minniti, on the 20th anniversary of their graduation at the University of Reggio Calabria (in Italy). %They are thankful to the anonymous referees of an earlier version of this work for their valuable suggestions.

\end{document}